\documentclass[10pt]{amsart}%{jarticle}
\usepackage{amsmath,amsthm, epsfig, amscd}
\usepackage{psfrag}
\usepackage[psamsfonts]{amssymb}
\usepackage[all]{xy}
\usepackage{mathrsfs}

\numberwithin{equation}{section}

\newtheorem{Thm}{Theorem}[section]

\newtheorem{Lem}[Thm]{Lemma}

\theoremstyle{remark}

\theoremstyle{definition}

\newcommand{\mysection}[2]{%
\vspace{2mm}\section{\bf #1}\label{#2}
}

\newcommand{\fig}[1]
        {\raisebox{-0.5\height}
                 {\includegraphics{#1}}
        }

\def\Z{{\mathbb Z}}
\def\R{{\mathbb R}}
\def\Q{{\mathbb Q}}

\def\calA{\mathscr{A}}

\def\calD{\mathscr{D}}

\def\calM{\mathscr{M}}

\def\deg{\mathrm{deg}}
\def\Hom{\mathrm{Hom}}

\newcommand{\mapright}[1]{
	\smash{\mathop{
		\hbox to 1cm{\rightarrowfill}}\limits^{#1}}}

\newcommand{\mapleft}[1]{
	\smash{\mathop{
		\hbox to 1cm{\leftarrowfill}}\limits^{#1}}}

\def\End{\mathrm{End}}
\def\Tr{\mathrm{Tr}}

\def\asum#1#2{\sum_{{{#1}\atop{#2}}}}
\def\Conf{C}

\def\bcalM{\overline{\calM}}

\def\ve{\varepsilon}
\def\bConf{\overline{C}}
\def\loc{\mathrm{local}}

\def\wLambda{\widehat{\Lambda}}
\def\wM{\widetilde{M}}
\def\wf{\widetilde{f}}
\def\wmu{\widetilde{\mu}}
\def\wxi{\widetilde{\xi}}
\def\anomaly{\mathrm{anomaly}}

\begin{document}

\title[A generalization of Fukaya's invariant of 3-manifolds II]{Higher order generalization of Fukaya's Morse homotopy invariant of 3-manifolds II. Invariants of 3-manifolds with $b_1=1$}
\author{Tadayuki Watanabe}
\address{Department of Mathematics, Shimane University,
1060 Nishikawatsu-cho, Matsue-shi, Shimane 690-8504, Japan}
\email{tadayuki@riko.shimane-u.ac.jp}
\date{\today}

{\noindent\footnotesize {\rm Preprint} (2016)}\par\vspace{15mm}
\maketitle
\vspace{-6mm}
\setcounter{tocdepth}{2}
\begin{abstract}
In this paper, it is explained that a topological invariant for 3-manifold $M$ with $b_1(M)=1$ can be constructed by applying Fukaya's Morse homotopy theoretic approach for Chern--Simons perturbation theory to a local coefficient system on $M$ of rational functions associated to the maximal free abelian covering of $M$. Our invariant takes values in Garoufalidis--Rozansky's space of Jacobi diagrams whose edges are colored by rational functions. It is expected that our invariant gives a lot of nontrivial finite type invariants of 3-manifolds.
\end{abstract}
\par\vspace{3mm}
%\tableofcontents

%%%%%%%%%%%%%%%%%%%%%%%%%%%%%%
%%%%%%%%%%%%%%%%%%%%%%%%%%%%%%
%%%%%%%%%%%%%%%%%%%%%%%%%%%%%%
\def\baselinestretch{1.06}\small\normalsize
\mysection{Introduction}{s:intro}

By using an idea of his Morse homotopy theory, K.~Fukaya constructed in \cite{Fu} a topological invariant of 3-manifolds with flat bundles on them that is analogous to Chern--Simons perturbation theory (\cite{AS, Ko1}). He considered a flat Lie algebra bundle over a 3-manifold $M$ and several Morse functions on $M$, and defined his invariant as the sum of the weights of some graphs (flow-graphs) in $M$ whose edges follow the gradients of the Morse functions. The weight is given by contracting the holonomies taken along the edges of a flow-graph by some tensor. Although Fukaya's construction was given for the 2-loop graphs, his construction also works for general 3-valent graphs at least when $M$ is a homology sphere with trivial connection (\cite{Wa1}). 

In this paper, we construct a topological invariant $\widetilde{z}_{2k}$ for closed oriented 3-manifolds $M$ with $b_1(M)=1$ by applying Fukaya's construction to a local coefficient system of rational functions associated to the maximal free abelian covering of $M$. There are fundamental results of Lescop about an equivariant 2-loop invariant for closed oriented 3-manifolds with $b_1(M)=1$ (\cite{Les1,Les2}), by which we were significantly influenced in the construction of $\widetilde{z}_{2k}$. The construction of $\widetilde{z}_{2k}$ would be rewritten by equivariant intersections in configuration spaces as given in \cite{Les1, Les2}. Prior to Lescop's works, Ohtsuki had given in \cite{Oh1, Oh2} a considerable refinement of the LMO invariant for 3-manifolds $M$ with $b_1(M)=1$, which is important in the study of equivariant perturbative invariant for non homology spheres. It is known that the LMO invariant (\cite{LMO}) is very strong for homology spheres whereas it is rather weaker for non homology spheres. It is remarkable that Ohtsuki's refined LMO invariant is also very strong for 3-manifolds $M$ with $b_1(M)=1$, and moreover his equivariant invariant is computable for some examples and yields some beautiful formulas. We expect that our invariant agrees with Ohtsuki's refined LMO invariant. 

In \cite{Wa2}, we construct an invariant of some degree 1 maps from 3-manifolds to the 3-torus by a method similar to the construction of this paper and apply it to study finite type invariants. It will follow from a result of \cite{Wa2} that the value of the 2-loop part $\widetilde{z}_2$ for some 3-manifolds with $b_1=1$ can be computed by clasper calculus of Goussarov and Habiro. In \cite{Wa3}, we define an invariant of fiberwise Morse functions on surface bundles over $S^1$, which can be considered as an analogue of the construction of the present paper for $S^1$-valued Morse theory. 

In the case where a knot in $M$ is present, a construction similar to that of this paper gives a knot invariant and gives many non-trivial finite type invariants of knots. We will explain this in a subsequent paper. 

Most of the construction in \cite{Wa1} is valid for the setting of this paper. We will only refer for what can be done by the same argument as \cite{Wa1}. 

%\tableofcontents

%\clearpage
%%%%%%%%%%%%%%%%%%%%%%%%%%%%%%
%%%%%%%%%%%%%%%%%%%%%%%%%%%%%%
\mysection{Preliminaries}{s:morse_complex}

\subsection{Acyclic Morse complex}

For simplicity, we assume that $M$ is an oriented, connected, closed 3-manifold with $b_1(M)=1$. Let $f:M\to \R$ be a Morse function and let $\Sigma\subset M$ be an oriented 2-submanifold that generates the oriented bordism group $\Omega_2(M)\cong H^1(M;\Z)=\Hom_\Z(H_1(M;\Z),\Z)\cong \Z$. By cutting $M$ along $\Sigma$ and by pasting its copies, the infinite cyclic covering $\pi:\widetilde{M}\to M$ is obtained. We denote by $t$ the generator of the group of covering transformation that shifts $\wM$ in the direction of the positive normal vector to $\Sigma$. Let $\wf:\wM\to \R$ be the Morse function that is the pullback $\wf=f\circ \pi$ and let $\wmu$ be the Riemannian metric on $\wM$ that is the pullback of the Riemannian metric $\mu$ on $M$. 

The pair $(\wf,\wmu)$ gives the gradient $\wxi$ on $\wM$. Let $P_i$ be the set of critical points of $f$ of index $i$ and we identify $P_i$ with the set of critical points of $\wf$ of index $i$ in a fundamental domain of $\wM$ between two successive components in $\pi^{-1}\Sigma$. Let $(C_*(\wf),\partial)$ be the Morse complex for $\wxi$ with $\Q$-coefficients. Namely, $C_i(\wf)=\Lambda^{P_i}$ ($\Lambda=\Q[t,t^{-1}]$) and the boundary $\partial:C_i(\wf)\to C_{i-1}(\wf)$ is given as follows.
\[ \partial(p)=\sum_{q\in P_{i-1}}\sum_{k\in\Z}n(\wxi;p,t^kq)\,t^kq,\]
where the coefficient $n(\wxi;p,t^kq)\in \Z$ is the number of the flow-lines of $-\wxi$ in $\wM$ that flow from $p$ to $t^kq$ counted with signs. In other words, the count of the flow-lines in $M$ from $p$ to $q$ whose intersection number with $\Sigma$ is $k$. The definition of the sign is given in \cite{Wa1}. We remark that the sum in the right hand side is finite. We put $\partial_{pq}=\sum_{k\in \Z}n(\wxi;p,t^kq)\,t^k$. It can be checked that $(C_*(\wf),\partial)$ is a chain complex, namely, $\partial^2=0$. The homology of the complex is identified with $H_*(\wM;\Q)$ as a $\Lambda$-module. 

Let $\wLambda$ be the field of rational functions $\frac{P(t)}{Q(t)}$ ($P(t),Q(t)\in\Lambda,\,Q(t)\neq 0$) and put
\[ C_i(f;\wLambda)=C_i(\wf)\otimes_\Lambda\wLambda=\wLambda^{P_i}. \]
This together with the boundary $\partial\otimes 1$ forms a chain complex. Since $\Lambda$ is a PID and $\wLambda$ is a flat $\Lambda$-module, we have, as $\Lambda$-modules,
\[ H_i(C_*(f;\wLambda))\cong H_i(\wM;\Q)\otimes_\Lambda \wLambda \] 
by the universal coefficient theorem (e.g., \cite[Theorem~VI.3.3]{CE}). From the fact that $\mathrm{rank}_\Lambda H_1(\wM;\Q)=0$ (e.g., \cite[Lemma~2.2]{Les1}) and by the Poincar\'{e} duality, it follows that $(C_*(f;\wLambda),\partial\otimes 1)$ is acyclic. 

\subsection{Combinatorial propagators}

$C=C_*(f;\wLambda)$ is a chain complex of based free $\wLambda$-modules. Let $\End_{\wLambda}(C)=\Hom_{\wLambda}(C,C)$ and let $\End_{\wLambda}(C)_k$ be its degree $k$ part. We define $\delta:\End_{\wLambda}(C)_k\to \End_{\wLambda}(C)_{k-1}$ by the following formula.
\[ \delta g=\partial\circ g-(-1)^k g\circ \partial. \]
This satisfies $\delta^2=0$. By the acyclicity of $C$ and by the K\"{u}nneth theorem (e.g., \cite[Theorem~VI.3.1a]{CE}), it follows that $(\End_{\wLambda}(C),\delta)$ is acyclic, too.

For example, $1\in \End_{\wLambda}(C)_0$ is a $\delta$-cycle. Thus there exists $g\in \End_{\wLambda}(C)_1$ such that $\delta g=\partial g+g\partial=1$. 
Such a $g$ is called a {\it combinatorial propagator} for $C$ (\cite{Fu}). For two choices $g,g'$ of combinatorial propagators for $C$, $g'-g$ is a $\delta$-cycle. Thus there exists $h\in\End_{\wLambda}(C)_2$ such that $\partial h-h\partial=g'-g$.

%\clearpage

%%%%%%%%%%%%%%%%%%%%%%%%%%%%%%
%%%%%%%%%%%%%%%%%%%%%%%%%%%%%%
\mysection{Perturbation theory with holonomies in $\wLambda$}{s:}

\subsection{Moduli space $\calM_\Gamma(\vec{\xi})$ of flow-graphs}

Let $f_1,f_2,\ldots,f_{3k}:M\to \R$ be a sequence of Morse functions and let $\xi_i$ be the gradient of $f_i$. We consider a connected edge-oriented trivalent graph with its sets of vertices and edges labelled and with $2k$ vertices and $3k$ edges. By the labelling $\{1,2,\ldots,3k\}\to \mathrm{Edges}(\Gamma)$ of $\Gamma$, we identify edges with numbers. Choose some of the edges and split each chosen edge into two arcs. We attach elements of $P_*(f_i)$ on the two 1-valent vertices ({\it white-vertices}) that appear after the splitting of the $i$-th edge. We call such obtained graph a {\it $\vec{C}$-graph} ($\vec{C}=(C_*(f_1;\wLambda),\ldots,C_*(f_{3k};\wLambda))$, see Figure~\ref{fig:cgraph}). A $\vec{C}$-graph has two kinds of ``edges'': a {\it compact edge}, which is connected,  and a {\it separated edge}, which consists of two arcs. We say that a separated edge obtained from a self-loop is {\it closed}. We call vertices that are not white vertices {\it black vertices}. If $p_i$ (resp. $q_i$) is the critical point attached on the input (resp. output) white vertex of a separated edge $i$, we define the degree of $i$ by $\deg(i)=\mathrm{ind}(p_i)-\mathrm{ind}(q_i)$, where $\mathrm{ind}(\cdot)$ denotes the Morse index. We define the degree of a compact edge $i$ by $\deg(i)=1$. We define the degree of a $\vec{C}$-graph by $\deg(\Gamma)=(\deg(1),\deg(2),\ldots,\deg(3k))$. 
\begin{figure}
\fig{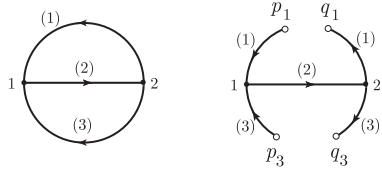}
\caption{A trivalent graph (left) and a $\vec{C}$-graph (right)}\label{fig:cgraph}
\end{figure}

We say that a continuous map $I$ from a $\vec{C}$-graph $\Gamma$ to $M$ is a {\it flow-graph} for the sequence $\vec{\xi}=(\xi_1,\xi_2,\ldots,\xi_{3k})$ if it satisfies the following conditions (see Figure~\ref{fig:flow-graph}).
\begin{figure}
\fig{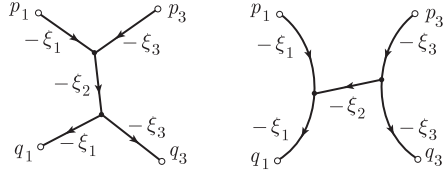}
\caption{Flow-graphs for $\vec{\xi}=(\xi_1,\xi_2,\xi_3)$}\label{fig:flow-graph}
\end{figure}

\begin{enumerate}
\item Every critical point $p_i$ attached on the $i$-th edge is mapped by $I$ to $p_i$ in $M$.
\item The restriction of $I$ to each edge of $\Gamma$ is a smooth embedding and at each point $x$ of the $i$-th edge that is not on a white vertex, the tangent vector of $I$ at $x$ (chosen along the edge orientation) is a positive multiple of $(-\xi_i)_x$. 
\end{enumerate}
For a $\vec{C}$-graph $\Gamma$, let $\calM_\Gamma(\vec{\xi})$ be the set of all flow-graphs for $\vec{\xi}$ from $\Gamma$ to $M$. By extracting black vertices, a natural map from $\calM_\Gamma(\vec{\xi})$ to the configuration space $\Conf_{2k}(M)$ of ordered tuples of $2k$ points is defined. It follows from a property of the gradient that this map is injective. This induces a topology on the set $\calM_\Gamma(\vec{\xi})$. 

\begin{Lem}[{Fukaya \cite{Fu,Wa1}}]\label{lem:0-mfd}
If $\vec{f}=(f_1,f_2,\ldots,f_{3k})$ and $\mu$ are generic, then for a $\vec{C}$-graph $\Gamma$ with $2k$ black vertices and $\deg(\Gamma)=(1,1,\ldots,1)$, the space $\calM_\Gamma(\vec{\xi})$ is a compact 0-dimensional manifold. Moreover, this property can be assumed for all $\vec{C}$-graphs with $2k$ black vertices simultaneously\footnote{In \cite{Wa1}, we considered flow-graphs on {\it punctured} homology sphere. Nevertheless, the proof of the corresponding lemma is essentially the same.}.
\end{Lem}

\subsection{The count of $\calM_\Gamma(\vec{\xi})$}

When the assumption of Lemma~\ref{lem:0-mfd} is satisfied, we may define an orientation of $\calM_\Gamma(\vec{\xi})$ in a similar way as \cite{Wa1}. Roughly, an orientation of $\calM_\Gamma(\vec{\xi})$ is defined as follows. The space $\calM_\Gamma(\vec{\xi})$ can be considered as the intersection of several smooth manifold strata in $M^{2k}$ each corresponds to the moduli space of an edge of $\Gamma$. We define an orientation of $\calM_\Gamma(\vec{\xi})$ by the coorientation $\bigwedge_{e\in\mathrm{Edges}(\Gamma)}v_e$ of $\calM_\Gamma(\vec{\xi})$ in $M^{2k}$ for some coorientations $v_e$ of the strata for $e$. If $e$ is compact or non-closed separated, then $e$ has two black vertices and $v_e$ is a vector in $\bigwedge^2(T_xM\oplus T_yM)$, where $x,y$ are the images from the black vertices of $e$.  If $e$ is closed separated, then the corresponding stratum is a flow-line $\gamma_{pq}$ of $-\xi_e$ between some pair $p,q$ of critical points with $\mathrm{ind}(p)=\mathrm{ind}(q)-1$. In this case, the orientation of $\gamma_{pq}$ is defined by $(-1)^{\mathrm{ind}(q)}\ve(\gamma_{pq})\,o$,\footnote{This definition is consistent with that for non-closed separated edges. Namely, this agees with that induced on the intersection of the stratum for a non-closed separated edge with the diagonal in $M\times M$. In the notation of \cite{Wa1}, the stratum for a separated edge $e$ is cooriented by $o^*_M(\calA_q(f_e))_x\wedge o^*_M(\calD_p(f_e))_y$, whereas $\gamma_{pq}$ is cooriented by $o^*_M(\calD_p(f_e))_x\wedge o^*_M(\calA_q(f_e))_x$. On the diagonal, the two differ by $(-1)^{\mathrm{ind}(q)(3-\mathrm{ind}(p))}=(-1)^{\mathrm{ind}(q)}$. } where $\ve(\gamma_{pq})$ is the sign of $\gamma_{pq}$ in the definition of $\partial(p)$, and $o$ is the orientation of $\gamma_{pq}$ given by $-\xi_e$.

In \cite{Wa1}, the number $\#\calM_\Gamma(\vec{\xi})$ was defined as the sum of the signs determined by the orientations. The following definition of $\#\calM_\Gamma(\vec{\xi})$ is different from that of \cite{Wa1}. We count points of $\calM_\Gamma(\vec{\xi})$ with weights in $\Lambda^{\otimes 3k}=\Q[t_1^{\pm 1},t_2^{\pm 1},\ldots,t_{3k}^{\pm 1}]$ as follows.
\[ \begin{split}
  \#\calM_\Gamma(\vec{\xi})&=\sum_{I\in\calM_\Gamma(\vec{\xi})}\ve(I),\quad \ve(I)=\pm t_1^{n_1}t_2^{n_2}\cdots t_{3k}^{n_{3k}},\\
  n_i &= (\mbox{The intersection number of the $i$-th edge of $I$ with $\Sigma$})
\end{split}\]
We take the sign $\pm$ as the one determined by the orientation of $I$. The intersection number of the $i$-th edge with $\Sigma$ is determined by the orientations of the edge and of $\Sigma$ and that of $M$, by $\mathrm{ori}(\mbox{$i$-th edge})\wedge \mathrm{ori}(\Sigma)=\mathrm{ori}(M)$.

\subsection{The generating series and its trace}

Let $R$ be either $\Lambda$ or $\wLambda$. An {\it $R$-colored $\vec{C}$-graph} is a pair of a $\vec{C}$-graph $\Gamma$ and a map $\phi:\mathrm{Edges}(\Gamma)\to R$. We will write an $R$-colored $\vec{C}$-graph as $\Gamma(\phi)$ or $\Gamma(\phi(1),\phi(2),\ldots,\phi(3k))$. We call $\phi$ an $R$-coloring of $\Gamma$. We define an action of a monomial $t_1^{n_1}t_2^{n_2}\cdots t_{3k}^{n_{3k}}$ on a $\vec{C}$-graph as follows.
\[ t_1^{n_1}t_2^{n_2}\cdots t_{3k}^{n_{3k}}\cdot \Gamma=\Gamma(t^{n_1},t^{n_2},\ldots,t^{n_{3k}}). \]
The right hand side is a $\Lambda$-colored $\vec{C}$-graph. By extending this by $\Q$-linearity, the formal linear combination $\#\calM_\Gamma(\vec{\xi})\cdot \Gamma$ of $\Lambda$-colored $\vec{C}$-graphs is defined. 

Let $\calA_{2k}(\Lambda)$ (resp. $\calA_{2k}(\widehat{\Lambda})$) be the vector space over $\Q$ spanned by pairs $(\Gamma,\phi)$, where $\Gamma$ is an (unlabelled) edge-oriented trivalent graph with $2k$ vertices and with vertex-orientation and $\phi$ is a $\Lambda$-coloring (resp. $\widehat{\Lambda}$-coloring) of $\Gamma$, quotiented by the relations AS, IHX, Orientation reversal, Linearity, Holonomy (Figure~\ref{fig:relations}) and automorphisms of oriented graphs\footnote{This definition is by Garoufalidis and Rozansky \cite{GR}. The AS and the IHX relations are due to Bar-Natan \cite{BN}.}. 
\begin{figure}
\fig{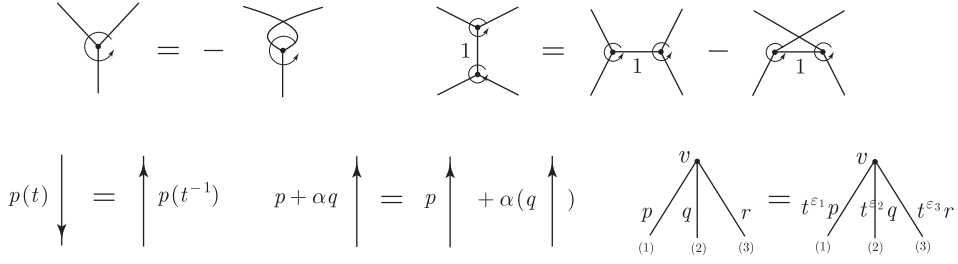}
\caption{The relations AS, IHX, Orientation reversal, Linearity and Holonomy. $p,q,r\in \Lambda$ (or $p,q,r\in \widehat{\Lambda}$), $\alpha\in\Q$. The exponent $\ve_i$ is $1$ if the $i$-th edge is oriented toward $v$ and otherwise $-1$.}\label{fig:relations}
\end{figure}

Now we shall define an element $z_{2k}(\vec{\xi})\in \calA_{2k}(\wLambda)$. Let $\vec{g}=(g^{(1)},g^{(2)},\ldots,g^{(3k)})$ be a sequence of combinatorial propagators for $\vec{C}=(C_*(f_1;\wLambda),\ldots,C_*(f_{3k};\wLambda))$. Then we define
\begin{equation}\label{eq:Z_k}
 z_{2k}(\vec{\xi})=\Tr_{\vec{g}}\Bigl(\sum_\Gamma \#\calM_\Gamma(\vec{\xi})\cdot \Gamma\Bigr). 
\end{equation}
Here, the sum is taken over all $\vec{C}$-graphs $\Gamma$ with $2k$ black vertices and $\deg(\Gamma)=(1,1,\ldots,1)$, and $\Tr_{\vec{g}}$ is defined as follows. For simplicity, we assume that the labels for the separated edges in a $\Lambda$-colored $\vec{C}$-graph $\Gamma(u_1(t),u_2(t),\ldots,u_{3k}(t))$ is $1,2,\ldots,r$. Let $p_i,q_i$ be the critical points on the input and output of the $i$-th edge of $\Gamma$, respectively and let $g_{q_ip_i}^{(i)}\in\wLambda$ be the coefficient of $p_i$ in $g^{(i)}(q_i)$. Then $\Tr_{\vec{g}}(\Gamma(u_1(t),u_2(t),\ldots,u_{3k}(t)))$ is the equivalence class in $\calA_{2k}(\wLambda)$ of a $\wLambda$-colored graph obtained by identifying each pair of the two white vertices of the separated edges in 
\[ \Gamma(-g^{(1)}_{q_1p_1}u_1(t),\ldots,-g^{(r)}_{q_rp_r}u_r(t),u_{r+1}(t),\ldots,u_{3k}(t)).\]
The definition of $\Tr$ can be generalized to graphs with other degrees in the same manner. 

\begin{Lem}\label{lem:L}
$z_{2k}(\vec{\xi})$ does not depend on the choices of $\vec{g}$ and of the hypersurface $\Sigma\subset M$ within the oriented bordism class.
\end{Lem}
\begin{proof}
That $z_{2k}(\vec{\xi})$ does not depend on the choice of $\vec{g}$ can be shown by the same argument as in \cite[\S{6}]{Wa1}.

For the rest, let $\Sigma'\subset M$ be another oriented 2-submanifold that is oriented bordant to $\Sigma$. Then by Morse theory, one may see that $\Sigma'$ is obtained from $\Sigma$ by a finite sequence of the following moves.
\begin{enumerate}
\item A homotopy in $M$.
\item An addition or a deletion of a 1-handle in a small ball $B$ in $M$.
\end{enumerate}
We may assume that for each move of type (2), the small ball $B$ is disjoint from all the critical points of $f_i$ and the graphs in $\bcalM_{\Gamma}(\vec{\xi})$ for all $\Gamma$. Thus a move of type (2) does not change $z_{2k}(\vec{\xi})$. 

The sum in (\ref{eq:Z_k}) may change under a move of type (1) when a homotopy intersects a black vertex of a flow-graph or intersects a critical point of some $f_i$. When a homotopy intersects a black vertex $v$, a $\Lambda$-coloring for the three edges incident to $v$ changes. The change of $\Lambda$-coloring is precisely the Holonomy relation. When a small homotopy intersects a critical point $p$ of some Morse function, say of $f_1$, the boundary operator of the twisted Morse complex $(C_*^{(1)},\partial^{(1)})$ for the first edge may change. Let $\overline{\partial}^{(1)}$ be the resulting boundary operator. Let $S_p:C_*^{(1)}\to C_*^{(1)}$ be the chain map of degree 0 defined for critical points $x\in P_*^{(1)}$ by
\[ S_p(x)=\left\{\begin{array}{ll}
t^{\pm 1}p & \mbox{if $x=p$}\\
x & \mbox{otherwise}
\end{array}\right. 
\]
where the sign $\pm 1$ depends on whether the homotopy crosses $p$ from above or below. Then we have $\overline{\partial}^{(1)}=S_p^{-1}\circ \partial^{(1)}\circ S_p$ and that $\overline{g}^{(1)}=S_p^{-1}\circ g^{(1)}\circ S_p:C_*^{(1)}\to C_{*+1}^{(1)}$ is a combinatorial propagator for $(C_*^{(1)},\overline{\partial}^{(1)})$.  There is an analogous left/right action of $S_p^{\pm 1}$ on a $\Lambda$-colored $\vec{C}$-graph given as follows. For a $\Lambda$-colored $\vec{C}$-graph $\Gamma(\phi)$, $\Gamma(\phi)\circ S_p$ (resp. $S_p\circ\Gamma(\phi)$) is the $\Lambda$-colored $\vec{C}$-graph obtained from $\Gamma(\phi)$ by replacing $\phi(1)$ with $t^{\pm 1}\phi(1)$ (resp. with $t^{\mp 1}\phi(1)$) if $p$ is the input (resp. the output) of the first edge of $\Gamma$ and otherwise $\Gamma(\phi)\circ S_p=\Gamma(\phi)$ (resp. $S_p\circ\Gamma(\phi)=\Gamma(\phi)$). After the small homotopy that crosses $p$, the flow graph that was counted as $\Gamma(\phi)$ will be counted as $S_p^{-1}\circ \Gamma(\phi)\circ S_p$. Now we have
\[ \Tr_{\overline{g}^{(1)},\ldots}(S_p^{-1}\circ \Gamma(\phi)\circ S_p)
=\Tr_{S_p\circ\overline{g}^{(1)}\circ S_p^{-1},\ldots}(\Gamma(\phi))
=\Tr_{g^{(1)},\ldots}(\Gamma(\phi)).
\]
This completes the proof of the invariance under a move of type (1).
\end{proof}

\subsection{The invariant $\widetilde{z}_{2k}$}

For the independence of the choice of $\vec{\xi}$, we shall define $\widetilde{z}_{2k}$ by adding a correction term to $z_{2k}(\vec{\xi})$. Though this could be done by the same method as \cite{Wa1}, the following definition by Shimizu (\cite{Sh1}) is nicer here. Take a compact oriented 4-manifold $W$ with $\partial W=M$ and with $\chi(W)=0$. By the condition $\chi(W)=0$, the outward normal vector field to $M$ in $TW|_{\partial W}$ can be extended to a nonsingular vector field $\nu_W$ on $W$. Let $T^vW$ be the orthogonal complement of the span of $\nu_W$. Then $T^vW$ is a rank 3 subbundle of $TW$ that extends $TM$. Take a sequence $\vec{\gamma}=(\gamma_1,\gamma_2,\ldots,\gamma_{3k})$ of generic sections of $T^vW$ so that $\gamma_i$ is an extension of $-\xi_i$. We define
\[ z_{2k}^\anomaly(\vec{\gamma})=\sum_\Gamma \#\calM_\Gamma^\loc(\vec{\gamma})\,[\Gamma(1,1,\ldots,1)] \in \calA_{2k}(\wLambda). \]
The sum is taken over all $\vec{C}$-graphs with $2k$ vertices and with only compact edges. Here, $\calM_\Gamma^\loc(\vec{\gamma})$ is the moduli space of affine graphs in the fibers of $T^vW$ whose $i$-th edge is a positive scalar multiple of $\gamma_i$. The number $\#\calM_\Gamma^\loc(\vec{\gamma})\in \Z$ is the count of the signs of the affine graphs that are determined by transversal intersections of some codimension 2 chains in a configuration space bundle over $W$. See \cite{Wa1} for detail. By the same argument as in \cite{Sh1}, it can be shown that $z_{2k}^\anomaly(\vec{\gamma})-\mu_k\,\mathrm{sign}\,W$, where $\mu_k\in\calA_{2k}(\wLambda)$ is the constant given in \cite{Wa1}, does not depend on the choices of $W$, $\nu_W$ and the extension $\vec{\gamma}$ of $\vec{\xi}$. We define $\widetilde{z}_{2k}(\vec{\xi})\in\calA_{2k}(\wLambda)$ by the following formula.
\[ \begin{split}
  \hat{z}_{2k}(\vec{\xi})&=z_{2k}(\vec{\xi})-z_{2k}^\anomaly(\vec{\gamma})+\mu_k\,\mathrm{sign}\,W,\\
  \widetilde{z}_{2k}(\vec{\xi})&=\sum_{\ve_i=\pm 1} \hat{z}_{2k}(\ve_1\xi_1,\ldots,\ve_{3k}\xi_{3k}).
\end{split} \]

\begin{Thm}\label{thm:1}
$\widetilde{z}_{2k}(\vec{\xi})$ is an invariant of the diffeomorphism type of $M$ and of the oriented bordism class of $\Sigma$.\footnote{There are only two possibilities for the class of $\Sigma$ that generates $\Omega_2(M)\cong\Z$, and the values of $\widetilde{z}_{2k}$ for the two differ by turning every $\wLambda$-coloring $\phi(t)$ on edge into $\phi(t^{-1})$. Thus, they can be considered essentially the same.}
\end{Thm}

%%%%%%%%%%%%%%%%%%%%%%%%%%%%%%
%%%%%%%%%%%%%%%%%%%%%%%%%%%%%%
\mysection{Proof of Theorem~\ref{thm:1}}{s:proof}

\subsection{Closedness of the 1-chain for a self-loop}

We will use the following lemma in the proof of Theorem~\ref{thm:1}. Let $\Gamma$ be a $\vec{C}$-graph whose $i$-th edge is a closed separated edge. Let $\Gamma'$ be the graph obtained from $\Gamma$ by removing the $i$-th edge. Then there is a 1-chain $O(\xi_i)$ of $M$ with untwisted $\wLambda$-coefficients\footnote{Notice that the holonomy along the closed separated edge does not depend on the position of the black vertex on it.} such that 
\begin{equation}\label{eq:M-O}
 \#\calM_\Gamma(\vec{\xi})=\calM_{\Gamma'}(\xi_1,\ldots,\widehat{\xi_i},\ldots,\xi_{3k})\cdot \mathrm{pr}_i^{-1}O(\xi_i) .
\end{equation}
Here, $\mathrm{pr}_i:M^{2k}\to M$ is the projection to the $i$-th factor, and the symbol $\cdot$ is the intersection between smooth manifold strata in $M^{2k}$, extended linearly to chains. Let $\calM_{p_iq_i}(\xi_i)$ denote the set of all flow-lines of $-\xi_i$ that flow from $p_i$ to $q_i$. Then $O(\xi_i)$ can be written as follows.
\begin{equation}\label{eq:O}
 O(\xi_i)=\sum_{{p_i,q_i}\atop{\mathrm{ind}(p_i)=\mathrm{ind}(q_i)+1}}\sum_{\gamma_{p_iq_i}\in\calM_{p_iq_i}(\xi_i)}(-1)^{\mathrm{ind}(q_i)}g_{q_ip_i}^{(i)}\ve(\gamma_{p_iq_i})\mathrm{Hol}(\gamma_{p_iq_i})\overline\gamma_{p_iq_i},
\end{equation}
where $\overline\gamma_{p_iq_i}$ is the 1-chain obtained from the flow-line by compactification, and $\mathrm{Hol}(\gamma_{p_iq_i})=t^{n_i}$, where $n_i$  is the intersection number of $\gamma_{p_iq_i}$ with $\Sigma$.

\begin{Lem}\label{lem:O-closed}
$O(\xi_i)+O(-\xi_i)$ is a 1-cycle\footnote{This lemma is implicit in \cite{Sh2}.}.  
\end{Lem}
\begin{proof} By (\ref{eq:O}), we have
\[ \partial O(\xi_i)=\sum_{{p_i,q_i}\atop{\mathrm{ind}(p_i)=\mathrm{ind}(q_i)+1}}\sum_{\gamma_{p_iq_i}\in\calM_{p_iq_i}(\xi_i)}(-1)^{\mathrm{ind}(q_i)}g_{q_ip_i}^{(i)}\ve(\gamma_{p_iq_i})\mathrm{Hol}(\gamma_{p_iq_i})(q_i-p_i). \]
The coefficient of $p_i$ of index $k$ in this formula is given by
\[ \begin{split}
  &-\sum_{{q_i}\atop{\mathrm{ind}(q_i)=k-1}}(-1)^{k-1}g_{q_ip_i}^{(i)}\partial_{p_iq_i}^{(i)}+\sum_{{r_i}\atop{\mathrm{ind}(r_i)=k+1}}(-1)^k g_{p_ir_i}^{(i)}\partial_{r_ip_i}^{(i)}\\
  &= (-1)^k (g^{(i)}\circ\partial^{(i)}+\partial^{(i)}\circ g^{(i)})_{p_ip_i}=(-1)^k.
\end{split} \]
Similarly, the coefficient of $p_i$ of index $3-k$ in $\partial O(-\xi_i)$ is given by
\[ \begin{split}
  &\sum_{{q_i}\atop{\mathrm{ind}(q_i)=3-(k-1)}}(-1)^{3-k}\overline{g}_{p_iq_i}^{(i)}\,\overline{\partial}_{q_ip_i}^{(i)}-\sum_{{r_i}\atop{\mathrm{ind}(r_i)=3-(k+1)}}(-1)^{3-(k+1)}\overline{g}_{r_ip_i}^{(i)}\overline{\partial}_{p_ir_i}^{(i)}\\
  &= -(-1)^k (g^{(i)*}\circ\partial^{(i)*}+\partial^{(i)*}\circ g^{(i)*})_{p_ip_i}=-(-1)^k,
\end{split} \]
where $\overline{g}^{(i)}(t)=g^{(i)}(t^{-1})$ etc., and $g^{(i)*}$ etc. is given by the adjoint matrix. This proves $\partial O(\xi_i)+\partial O(-\xi_i)=0$.
\end{proof}

\subsection{Completing the proof of Theorem~\ref{thm:1}}

When the $i$-th edge of a $\vec{C}$-graph $\Gamma$ is a separated edge on which $x,y\in P^{(i)}$ are attached on the input/output respectively, we will write $\Gamma=\Gamma(x,y)_i$. This notation enables us to express the graph $\Gamma(x,y)_i$ with $x,y$ replaced with $x',y'$ respectively, as $\Gamma(x',y')_i$. The notation $\Gamma(\emptyset,\emptyset)_i$ will denote the graph obtained from $\Gamma(x,y)_i$ by replacing the $i$-th edge with a compact edge.

Since the proof is parallel to that of the main theorem of \cite{Wa1}, we only give an outline. We show that the value of $\widetilde{z}_{2k}$ does not change if one Morse function in the sequence $\vec{f}$, say $f_1$, is replaced with another Morse function $f_1'$. As usual in Cerf theory (\cite{Ce}), we use the fact that there is a smooth 1-parameter family $\{h_s:M\to \R\}_{s\in [0,1]}$ that restricts to $f_1$ and $f_1'$ on $s=0,1$ respectively, such that $h_s$ is Morse except for finitely many values of $s$ and at the excluded values the singularities of $h_s$ consist of birth-death singularities and Morse singularities. Moreover, there may be finitely many values of $s$ at which the Morse complex for $\wxi_s=\mathrm{grad}\,\widetilde{h}_s$ changes, namely, at which there is a flow-line of $\xi_s=\mathrm{grad}\,h_s$ between two Morse critical points of the same index $j$. Such a flow-line is called a {\it $j/j$-intersection} and corresponds to a handle-slide.

Let $J=[s_0,s_1]\subset [0,1]$ be an interval that does not have birth-death parameter. By replacing $f_1$ with the family $\{h_s\}$, the moduli spaces $\calM_\Gamma(\vec{\xi}_J)$ and its natural compactifications $\bcalM_\Gamma(\vec{\xi}_J)$ for flow-graphs mapped along the fiber $M$ in $J\times M$ are defined. The moduli space $\bcalM_\Gamma(\vec{\xi}_J)$ is an oriented compact 1-dimensional manifold immersed in $J\times \bConf_{2k}(M)$, where $\bConf_{2k}(M)$ is the differential geometric analogue of the Fulton--MacPherson compactification of the configuration space $\Conf_{2k}(M)$ (\cite{AS, Ko1}). If $\bcalM_\Gamma(\vec{\xi}_J)$ does not have boundaries except the endpoints of $J$ for every $\Gamma$, then it gives a cobordism between the moduli spaces on the endpoints of $J$ and it follows that the value of $z_{2k}$ does not change between $s_0$ and $s_1$. Here, the value of $\#\calM_\Gamma(\vec{\xi}_s)$, $\vec{\xi}_s=(\xi_s,\xi_2,\ldots,\xi_{3k})$, may change when a vertex of $\Gamma$ intersects $\Sigma$, but the difference is killed by the $\Tr$ as in Lemma~\ref{lem:L} or by the Holonomy relation and the trace is invariant. 

In general, $\bcalM_\Gamma(\vec{\xi}_J)$ may have boundaries on the interior of $J$. It follows from results in \cite[\S{8}]{Wa1} that the boundary of $\bcalM_\Gamma(\vec{\xi}_J)$ consists of degenerate flow-graphs as follows (see Figure~\ref{fig:degenerate-graphs}). 

\begin{enumerate}
\item A subgraph (or the whole) of $\Gamma$ collapses into a point of $M$.
\item An edge of $\Gamma$, either compact or separated, splits in the middle by a critical point.
\item The black vertex on a closed separated edge coincides with a critical point.
\end{enumerate}
\begin{figure}
\fig{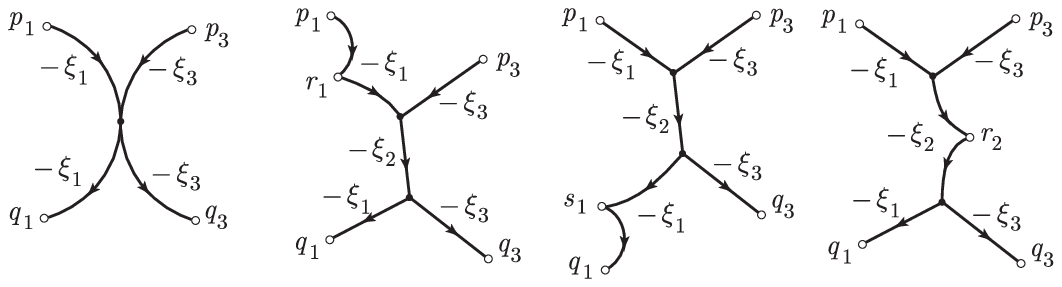}\\
\fig{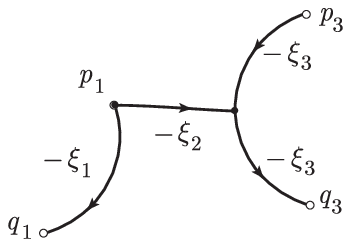}
\caption{Degenerate flow-graphs}\label{fig:degenerate-graphs}
\end{figure}

Among the degenerations of type (1), if the whole of a graph collapses, then $z_{2k}$ may change. However, $\hat{z}_{2k}$ does not change because the change is cancelled by the change of the correction term, as shown in \cite[\S{10}]{Wa1}. When a proper subgraph with at least 3 vertices collapses, then the sum of the changes is shown to vanish by the same arguments of symmetries as given in \cite{Ko1}. When a subgraph with exactly 2 vertices collapses, the sum of the changes vanishes by the IHX relation. See \cite[\S{5, 6}]{Wa1} for detail about this paragraph. 

The degenerations of type (3) do not change $\widetilde{z}_{2k}$ by (\ref{eq:M-O}) and Lemma~\ref{lem:O-closed}.

The degenerations of type (2) can be treated by the same argument as in \cite[\S{10}]{Wa1}. There are no changes in the proof except that $\Z$ is replaced with $\Lambda$ and that the coefficients of graphs belong to $\Lambda^{\otimes 3k}$. Namely, around a parameter $s_0\in J$ of the inner boundary of $\bcalM_\Gamma(\vec{\xi}_J)$ where there are no $j/j$-intersections, the difference $\hat{z}_{2k}(\vec{\xi}_{s_0+\ve})-\hat{z}_{2k}(\vec{\xi}_{s_0-\ve})$ is given by $\Tr_{\vec{g}}$ of the terms of degenerate flow-graphs of type (2) and it is equal to $\Tr_{\vec{g}}$ of 
\[ \begin{split}
  -\sum_{i=1}^{3k}\asum{\Gamma'(\tilde{p}_i,\tilde{q}_i)_i}{\mathrm{ind}(\tilde{p}_i)=\mathrm{ind}(\tilde{q}_i)}W_{\Gamma'(\tilde{p}_i\tilde{q}_i)_i}\cdot
\Bigl(
&\asum{x_i\in P_*^{(i)}}{\mathrm{ind}(x_i)=\mathrm{ind}(\tilde{p}_i)+1}\partial^{(i)}_{x_i\tilde{p}_i}\cdot\Gamma'(x_i,\tilde{q}_i)_i\\
+&\asum{y_i\in P_*^{(i)}}{\mathrm{ind}(y_i)=\mathrm{ind}(\tilde{q}_i)-1}\partial^{(i)}_{\tilde{q}_iy_i}\cdot\Gamma'(\tilde{p}_i,y_i)_i+\delta_{\tilde{p}_i\tilde{q}_i}\Gamma'(\emptyset,\emptyset)_i
\Bigr),
\end{split} \]
where the second sum is taken over uncolored graphs of the form $\Gamma'(\tilde{p}_i,\tilde{q}_i)_i$ of degree $(\eta_1,\eta_2,\ldots,\eta_{3k})$ with $\eta_\ell=1$ for $\ell\neq i$ and $\eta_i=0$, and 
\[ W_{\Gamma'(\tilde{p}_i\tilde{q}_i)_i}=(-1)^{\mathrm{ind}(\tilde{p}_i)}\#\calM_{\Gamma'(\tilde{p}_i,\tilde{q}_i)_i}(\vec{\xi}_J)\in\Lambda^{\otimes 3k}.\]
We consider $\partial^{(i)}_{x_i\tilde{p}_i}$ etc. as an element of $\Lambda^{\otimes 3k}$ by identifying $\Lambda$ with $1^{\otimes (i-1)}\otimes \Lambda\otimes 1^{\otimes (3k-i)}$. For each fixed pair $\tilde{p}_i,\tilde{q}_i\in P^{(i)}_*$ with $\mathrm{ind}(\tilde{p}_i)=\mathrm{ind}(\tilde{q}_i)$, we have
\[ \begin{split}
&\Tr_{\vec{g}}\Bigl(W_{\Gamma'(\tilde{p}_i\tilde{q}_i)_i}\cdot\Bigl(\asum{x_i\in P_*^{(i)}}{\mathrm{ind}(x_i)=\mathrm{ind}(\tilde{p}_i)+1}\partial^{(i)}_{x_i\tilde{p}_i}\cdot\Gamma'(x_i,\tilde{q}_i)_i\\
&\hspace{22mm}+\asum{y_i\in P_*^{(i)}}{\mathrm{ind}(y_i)=\mathrm{ind}(\tilde{q}_i)-1}\partial^{(i)}_{\tilde{q}_iy_i}\cdot\Gamma'(\tilde{p}_i,y_i)_i+\delta_{\tilde{p}_i\tilde{q}_i}\Gamma'(\emptyset,\emptyset)_i
\Bigr)\Bigr) \\
&=\Tr_{\ldots,\partial^{(i)}g^{(i)}+g^{(i)}\partial^{(i)},\ldots}\Bigl(W_{\Gamma'(\tilde{p}_i\tilde{q}_i)_i}\cdot\Gamma'(\tilde{p}_i,\tilde{q}_i)_i\Bigr)+\Tr_{\vec{g}}\Bigl(W_{\Gamma'(\tilde{p}_i\tilde{q}_i)_i}\cdot\delta_{\tilde{p}_i\tilde{q}_i}\Gamma'(\emptyset,\emptyset)_i\Bigr)\\
&=\Tr_{\ldots,\mathrm{id},\ldots}\Bigl(W_{\Gamma'(\tilde{p}_i\tilde{q}_i)_i}\cdot\Gamma'(\tilde{p}_i,\tilde{q}_i)_i\Bigr)+\Tr_{\vec{g}}\Bigl(W_{\Gamma'(\tilde{p}_i\tilde{q}_i)_i}\cdot\delta_{\tilde{p}_i\tilde{q}_i}\Gamma'(\emptyset,\emptyset)_i\Bigr)=0.
\end{split}\]
Around a parameter $s_0\in J$ of the inner boundary of $\bcalM_\Gamma(\vec{\xi}_J)$ at a $j/j$-intersection, the difference $\hat{z}_{2k}(\vec{\xi}_{s_0+\ve})-\hat{z}_{2k}(\vec{\xi}_{s_0-\ve})$ is decomposed into two parts, as follows. Let $g,g'$ be the combinatorial propagators for $C^{(1)}$ at $s_0-\ve$ and $s_0+\ve$ respectively and let $\vec{g}=(g,g^{(2)},\ldots,g^{(3k)}),\vec{g}'=(g',g^{(2)},\ldots,g^{(3k)})$. Then we have
\[ \begin{split}
	&\sum_{\Gamma}\Tr_{\vec{g}'}\bigl(\#\bcalM_\Gamma(\vec{\xi}_{s_0+\ve})\cdot \Gamma\bigr)
	-\sum_{\Gamma}\Tr_{\vec{g}}\bigl(\#\bcalM_\Gamma(\vec{\xi}_{s_0-\ve})\cdot \Gamma\bigr)\\
	&=\sum_\Gamma\Tr_{\vec{g}'}\Bigl(\bigl(\#\bcalM_\Gamma(\vec{\xi}_{s_0-\ve})-\#\bcalM_{d''\Gamma}(\vec{\xi}_J)\bigr)\cdot\Gamma\Bigr)-\sum_\Gamma\Tr_{\vec{g}}\bigl(\#\bcalM_\Gamma(\vec{\xi}_{s_0-\ve})\cdot\Gamma\bigr)\\
	&=\sum_\Gamma\Tr_{g'-g,\ldots}\bigl(\#\bcalM_\Gamma(\vec{\xi}_{s_0-\ve})\cdot\Gamma\bigr)-\sum_\Gamma\Tr_{\vec{g}'}\bigl(\#\bcalM_{d''\Gamma}(\vec{\xi}_J)\cdot\Gamma\bigr).
\end{split}\]
Here, $\#\bcalM_{d''\Gamma}(\vec{\xi}_J)$ is the sum of the counts of the degenerate flow-graphs including the $j/j$-intersection with appropriate signs.  The first term in the last line corresponds to the change of the combinatorial propagator and the other one corresponds to the count of the degenerate flow-graphs including the $j/j$-intersection. The change of the combinatorial propagator can be described explicitly as follows. The underlying $\wLambda$-modules $\vec{C}$ do not change between $s_0-\ve$ and $s_0+\ve$ while the boundary operator $\partial^{(1)}$ may change at $s_0$ and the combinatorial propagator $g$ changes accordingly. For an endomorphism $1+h\in \End_{\wLambda}(C^{(1)})_0$ corresponding to an elementary matrix, where $h$ counts the $j/j$-intersection with holonomy, $g'$ can be given as $g'=(1+h)\circ g\circ (1-h)$, which gives $g'-g=hg'-g'h$. The trace of a graph by $hg'-g'h$ is cancelled by the part of the counts of the degenerate flow-graphs including the $j/j$-intersection. See \cite[\S{10}]{Wa1} for detail.

When $s$ crosses a parameter of a birth-death singularity, a separated edge will be glued together into a compact edge, or its reverse. The explicit form of the gluing is exactly the same as \cite[\S{8}]{Wa1}, because the gluing is local. 
\qed

\section*{\bf Acknowledgments.}
I would like to thank Tatsuro Shimizu for explaining to me his works. 
This work is supported by JSPS Grant-in-Aid for Scientific Research 26800041 and 26400089.
\par

\end{document}